\documentclass[a4paper,11pt,oneside]{amsart}
\usepackage[left=2.7cm,right=2.7cm,top=3.5cm,bottom=3cm]{geometry}

\usepackage[all]{xy}
\usepackage{amscd}
\usepackage{amsfonts}
\usepackage{amssymb}
\usepackage{amsmath}
\usepackage{latexsym}
\usepackage[usenames,dvipsnames]{color}

\numberwithin{equation}{section}

\newtheorem{thm}{Theorem}[section]
\newtheorem{cor}[thm]{Corollary}
\newtheorem{lem}[thm]{Lemma}
\newtheorem{prop}[thm]{Proposition}

\theoremstyle{definition}
\newtheorem{df}[thm]{Definition}
\newtheorem{rem}[thm]{Remark}
\newtheorem{rems}[thm]{Remarks}
\newtheorem{ass}[thm]{Assumption}
\newtheorem{exem}[thm]{Example}

\newcommand{\Q}{\mathbb{Q}}

\newcommand{\N}{\mathbb{N}}
\newcommand{\Z}{\mathbb{Z}}

\newcommand{\F}{\mathbb{F}}
\newcommand{\G}{\Gamma}

\renewcommand{\L}{\Lambda}

\renewcommand{\geq}{\geqslant}

\newcommand{\Gal}{\textrm{Gal}}

\newcommand{\idl}{{\bf I}}

\newcommand{\liminv}{\displaystyle \lim_{\leftarrow}}

\newcommand{\ol}{\mathcal O}
\newcommand{\cala}{\mathcal A}

\newcommand{\calm}{\mathcal M}

\newcommand{\call}{\mathcal L}

\newcommand{\calf}{\mathcal F}

\newcommand{\pr}{\mathfrak p}

\newcommand{\gotm}{\mathfrak m}
\newcommand{\gotq}{\mathfrak q}
\newcommand{\gotP}{\mathfrak P}

\newcommand{\sri}{\twoheadrightarrow}
\newcommand{\iri}{\hookrightarrow}
\newcommand{\ri}{\rightarrow}

\newcommand{\wt}{\widetilde}
\newcommand{\wh}{\widehat}

\newcommand{\il}[1]{\lim_{\buildrel \longleftarrow\over{#1}}}

\newcommand{\Hom}{\operatorname{Hom}}

\title[Characteristic ideals and Iwasawa theory]{Characteristic ideals and Iwasawa theory}

\author[A. Bandini] {Andrea Bandini}
\address{Dipartimento di Matematica e Informatica, Universit\`a degli Studi di Parma\\
Parco Area delle Scienze, 53/A - 43124 Parma (PR), Italy}
\email{andrea.bandini@unipr.it}

\author[F. Bars] {Francesc Bars}
\address{Departament Matem\`atiques, Edif. C, Universitat Aut\`onoma de Barcelona\\
08193 Bellaterra, Catalonia}
\email{francesc@mat.uab.cat}
\thanks{F. Bars supported by MTM2013-40680-P}

\author[I. Longhi] {Ignazio Longhi}
\address{Department of Mathematical Sciences,
Xi'an Jiaotong-Liverpool University \\
111 Ren Ai Road,
Dushu Lake Higher Education Town\\
Suzhou Industrial Park, Suzhou, Jiangsu\\
215123, China}
\email{Ignazio.Longhi@xjtlu.edu.cn}
\thanks{I. Longhi supported by National Science Council of Taiwan, grant NSC100-2811-M-002-079}

\keywords{Characteristic ideals; Iwasawa theory; Krull rings; class groups}

\subjclass[2010]{11R23; 13F25}

\begin{document}

\begin{abstract}
Let $\L$ be a non-noetherian Krull domain which is the inverse limit of noetherian Krull domains $\L_d$
and let $M$ be a finitely generated $\L$-module which is the inverse limit of $\L_d$-modules $M_d\,$.
Under certain hypotheses on the rings $\L_d$ and on the
modules $M_d\,$, we define a pro-characteristic ideal for $M$ in $\L$, which should play
the role of the usual characteristic ideals for finitely generated modules over noetherian
Krull domains. We apply this to the study of Iwasawa modules (in particular of class groups) in a
non-noetherian Iwasawa algebra $\Z_p[[\Gal(\calf/F)]]$, where $F$ is a function field of characteristic
$p$ and $\Gal(\calf/F)\simeq\Z_p^\infty$.
\end{abstract}

\maketitle

\tableofcontents

\section{Introduction}
Let $A$ be a noetherian Krull domain and $M$ a finitely generated torsion $A$-module. The structure theorem for such modules provides
an exact sequence
\begin{equation} \label{e:1} 0 \longrightarrow P \longrightarrow M \longrightarrow \bigoplus_{i=1}^n A/\pr_i^{e_i} A
\longrightarrow Q \longrightarrow 0 \end{equation}
where the $\pr_i$'s are height 1 prime ideals of $A$ and $P$ and $Q$ are pseudo-null $A$-modules (for more details and precise
definitions of all the objects appearing in this Introduction see Section \ref{CharIdSec}). This sequence defines an important
invariant of the module $M$, namely its {\em characteristic ideal}
\[ Ch_A(M):= \prod_{i=1}^n \pr_i^{e_i} \ .\]

Characteristic ideals play a major role in (commutative) Iwasawa theory for global fields: they provide the algebraic counterpart
for the $p$-adic $L$-functions associated to Iwasawa modules (such as class groups or duals of Selmer groups).
Here the Krull domain one works with is the Iwasawa algebra $\Z_p[[\Gamma]]$, where $\Gamma$
is a commutative $p$-adic Lie group occurring as Galois group (we shall deal mainly with the case $\Gamma\simeq \Z_p^d$ for some
$d\in \mathbb{N}$). Even if pseudo-null modules do not contribute to characteristic ideals, they appear in the descent problem
when one wants to compare the characteristic ideal of an Iwasawa module of a $\Z_p^d$-extension with the one of a
$\Z_p^{d-1}$-extension contained in it. The last topic is particularly important when the global field has characteristic
$p$, because, in this case, extensions $\calf/F$ with $\Gal(\calf/F)\simeq \Z_p^\infty$ occur quite naturally: in this
situation the Iwasawa algebra is non-noetherian and there is no guarantee one can find a sequence such as \eqref{e:1}.
One strategy to overcome this difficulty is to consider a filtration of $\Z_p^d$-extensions for $\calf$, define
the characteristic ideals at the $\Z_p^d$-level for all $d$ and then pass to the limit.

To deal with the technical complications of inverse limits and projections of pseudo-null modules, in Section \ref{CharIdSec} we
prove the following (see Propositions \ref{pseudonull} and \ref{CharId1})

\begin{prop}\label{IntroProp}
Let $A$ be  a noetherian Krull domain and put $B:=A[[t]]$. If $M$ is a pseudo-null $B$-module, then $M_t$
(the kernel of multiplication by $t$) and $M/tM$ are finitely generated torsion $A$-modules and
\begin{equation}\label{Intro1} Ch_A(M_t)=Ch_A(M/tM) \ .\end{equation}
Moreover, for any finitely generated torsion $B$-module $N$, we have
\begin{equation}\label{Intro2} Ch_A(N_t)\pi(Ch_B(N)) = Ch_A(N/tN) \end{equation}
(where $\pi\colon B\rightarrow A$ is the canonical projection).
\end{prop}

\noindent This immediately provides a criterion for an $B$-module to be pseudo-null (Corollary \ref{EqVerCharId}), but
our main application is the definition of an analogue of characteristic ideals in a non-noetherian Iwasawa algebra
(Theorem \ref{MainThm}).

\begin{thm}\label{IntroThm}
Let $\{\L_d\}_{d\geq 0}$ be an inverse system of noetherian Krull domains such that
\[ \L_d \simeq \L_{d+1}/\pr_{d+1}\ \  and\ \  \L_{d+1} \simeq \il{n} \L_{d+1}/\pr_{d+1}^n\ \
for\ any\ d\geq 0 \]
($\pr_{d+1}$ a principal prime ideal of $\L_{d+1}$ of height 1).
Let $\L:=\displaystyle{\il{d} \L_d}$ and consider a finitely generated $\L$-module
$M= \displaystyle{\il{d} M_d}$ (where each $M_d$ is a $\L_d$-module).  If, for any $d \gg 1$, \begin{itemize}
\item[{\bf 1.}] the $\pr_d$-torsion submodule of $M_d$ is a pseudo-null $\L_{d-1}$-module, i.e.,
$Ch_{\L_{d-1}}(M_d[\pr_d])=(1)$;
\item[{\bf 2.}] $Ch_{\L_{d-1}}(M_d/\pr_d)\subseteq Ch_{\L_{d-1}}(M_{d-1})$,
\end{itemize}
then we can define a {\em pro-characteristic ideal} for $M$ as
\[ \wt{Ch}_\L(M):=\il{d}\, (\pi^\L_{\L_d})^{-1}(Ch_{\L_d}(M_d)) \]
(where $\pi^\L_{\L_d}\,: \L\rightarrow \L_d$ is the natural map provided by the inverse limit defining $\L$).
\end{thm}

The Iwasawa algebra associated with a $\Z_p^d$-extension of a global field $F$ is (noncanonically) isomorphic to the Krull
ring $\Z_p[[t_1,\dots,t_d]]$, hence descending to a $\Z_p^{d-1}$-extension corresponds to the passage from $A[[t]]$ to $A$.
So the results of Section \ref{CharIdSec} apply immediately to Iwasawa modules and, in order to keep the paper short, we
just consider the case of class groups. We also remark that the analogue of Proposition \ref{IntroProp} for these
Iwasawa algebras has been proved by T. Ochiai in \cite[Section 3]{Oc}
and it suffices for the arithmetical applications we had in mind. Indeed in Section \ref{ClGrSec} we deal with a
global field $F$ of characteristic $p>0$ (for an account of Iwasawa theory over function fields see \cite{BBL} and the
references there). We already mentioned that here one is naturally led to work with extremely large
abelian $p$-adic extensions: this comes from class field theory, since, in the completion $F_v$ of $F$ at some place $v$, the
group of 1-units $U_1(F_v)$ is isomorphic to $\Z_p^\infty\,$. As hinted above, our strategy to tackle $\calf/F$ with
$\Gal(\calf/F)\simeq \Z_p^\infty$ is to work first with $\Z_p^d$-extensions and then use limits. The usefulness of Theorem
\ref{IntroThm} in this procedure is illustrated in Section \ref{PrRamCGSubs}, where we define the pro-characteristic ideal
$\wt{Ch}_{\L}(\cala(\calf))$ dispensing with the crutch of the {\em ad hoc} hypothesis \cite[Assumption 5.3]{BBL}.
The search for a ``good'' definition for it was one of the main motivations for this work.

\noindent The arithmetic significance of our pro-characteristic ideal is ensured by a deep result of D. Burns (see \cite[Theorem 3.1]{blt09}
and the Appendix \cite{blt09A}), which shows that the characteristic ideal of the class group of a $\Z_p^d$-extension $\calf_d/F$ is generated
by a Stickelberger element (by some language abuse we shall call {\em class group of $\calf_d$} the inverse limit of the class groups
of the finite subextensions of $\calf_d/F$). Therefore (see Corollary \ref{IMCClassGr}) our pro-characteristic ideal is generated by a
Stickelberger element as well and this can be considered as an instance of Iwasawa Main Conjecture for non-noetherian Iwasawa
algebras.

Next to class groups, \cite{BBL} and \cite{BL} consider the case of Selmer groups of abelian varieties: in \cite[Section 3]{BBL}
we employed Fitting ideals of Pontrjagin duals of Selmer groups instead than characteristic ideals in order to avoid the difficulties
of taking the inverse limit. With some additional work, Theorem \ref{IntroThm} permits to define
a pro-characteristic ideal for these modules as well, allowing to formulate a more classical Iwasawa Main Conjecture
(details can be found in \cite{BBL2}).

\begin{rem}\label{IntroRem}
If a pseudo-null $A[[t]]$-module $M$ is finitely generated over $A$ as well, then the statement of Proposition
\ref{IntroProp} is trivially deduced from the exact sequence
\[ 0 \ri M_t \ri M \buildrel{t}\over\longrightarrow M \ri M/tM \ri 0 \]
and the multiplicativity of characteristic ideals. As explained in \cite{Gr4} (see Lemma 2 and the discussion right after it),
for any $\Z_p^d$-extension $\calf_d/F$ and any pseudo-null $\L(\calf_d)$-module $M$ it is always possible to find (at least one)
$\Z_p^{d-1}$-subextension $\calf_{d-1}$ such that $M$ is finitely generated over $\L(\calf_{d-1})$ (where $\L(\mathcal{L})$ is
the Iwasawa algebras associated with the extension $\mathcal{L}/F$). Our search for a
characteristic ideal via a projective limit does not allow this freedom in the choice of subextensions, hence the need for an
``unconditional'' result like Theorem \ref{IntroThm}.
\end{rem}

\section{Pseudo-null modules and characteristic ideals}\label{CharIdSec}

\subsection{Krull domains} We begin by reviewing some basic facts and definitions we are going to need.
A comprehensive reference is \cite[Chapter VII]{Bo}.

An integral domain $A$ is called a Krull domain if $A=\cap A_\pr$ (where $\pr$ varies among prime ideals of height 1 and
$A_\pr$ denotes localization), all $A_\pr$'s are discrete valuation rings and any $x\in A-\{0\}$ is a unit in $A_\pr$ for almost all $\pr$
\footnote{This is not the definition in \cite{Bo}, but it is equivalent to it: see \cite[VII, \S 1.6, Theorem 4]{Bo}.}.
In particular, one attaches a discrete valuation to any height 1 prime ideal. Furthermore, a ring is a unique
factorization domain if and only if it is a Krull domain and all height 1 prime ideals are principal (\cite[VII, \S 3.2, Theorem 1]{Bo}).

\subsubsection{Torsion modules}
Let $A$ be a noetherian Krull domain. A finitely generated torsion $A$-module is said to be {\em pseudo-null} if its annihilator
ideal has height at least 2.
A morphism with pseudo-null kernel and cokernel is called a pseudo-isomorphism: being pseudo-isomorphic is an equivalence relation
between finitely generated torsion $A$-modules (torsion is essential here \footnote{For example the map $(p,t)\iri \Z_p[[t]]$ is a
pseudo-isomorphism, but there is no such map from $\Z_p[[t]]$ to $(p,t)$.})
and we shall denote it by $\sim_A\,$. If $M$ is a finitely generated torsion $A$-module then there is a pseudo-isomorphism
\begin{equation}\label{StrThEq} M\longrightarrow  \bigoplus_{i=1}^n A/\pr_i^{e_i} \end{equation}
where the $\pr_i$'s are height 1 prime ideals of $A$ (not necessarily distinct) and the $\pr_i$'s, $n$ and the
$e_i$'s are uniquely determined by $M$ (see e.g. \cite[VII, \S 4.4, Theorem 5]{Bo}). A module like the one on the right-hand side
of \eqref{StrThEq} will be called {\em elementary $A$-module} and
\[ E(M) :=  \bigoplus_{i=1}^n A/\pr_i^{e_i} \sim_A M \]
is the {\em elementary module attached to $M$}.

\begin{df}\label{defCharid}
Let $M$ be a finitely generated $A$-module: its {\em characteristic ideal} is
\[ Ch_A(M):=\left\{ \begin{array}{ll}
0 & \text{if } M \text{ is not torsion;} \\
{\displaystyle{\prod_{i=1}^n \pr_i^{e_i}}} & \text{if } M\sim_A \displaystyle{\bigoplus_{i=1}^n A/\pr_i^{e_i}}\,.
\end{array} \right.\]
In particular, $M$ is pseudo-null if and only if $Ch_A(M)=A$.
\end{df}

We shall denote by ${\bf Fgt}_A$ the category of finitely generated torsion $A$-modules.

\begin{rems}\label{supernatpseudoisom}
\begin{itemize} \item[{}]
\item[{\bf 1.}] An equivalent definition of pseudo-null is that all localizations at primes of height 1 are zero. If $\pr$ and
$\mathfrak q$ are two different primes of height 1 (and $M$ is a torsion $A$-module) we have $M\otimes_A A_\pr\otimes_A A_{\mathfrak q}=0$.
By the structure theorem recalled in \eqref{StrThEq} it follows immediately that for a finitely generated torsion $A$-module $M$,
the canonical map
\begin{equation} \label{eqpseudoisom} M\longrightarrow \bigoplus_\pr \big(M\otimes_A A_\pr\big) \end{equation}
(where the sum is taken over all primes of height 1) is a pseudo-isomorphism. Actually, the right-hand side of
\eqref{eqpseudoisom} can be used to compute $Ch_A(M)$: a prime $\pr$ appears in $Ch_A(M)$ with exponent the length of the module $M\otimes_A A_\pr$.
\item[{\bf 2.}] The previous remark suggests a generalization of the definition of characteristic ideal by means of {\em supernatural divisors}
\footnote{We recall that the group of divisors of $A$ is the free abelian group generated by the prime ideals of height 1 in $A$
(see \cite[VII, \S 1]{Bo}).}.
Let $M$ be any torsion $A$-module (we drop the finitely generated assumption) and define
$$Ch_A(M):=\prod_\pr\pr^{l_\pr(M\otimes_A A_\pr)}$$
where the product is taken over all primes of height 1 and the exponent of $\pr$ (i.e., the {\em length} of the module $M\otimes_A A_\pr\,$)
is a supernatural number (i.e., belongs to $\N\cup\{\infty\}\,$). More precisely, for $N$ a finitely generated torsion $A_\pr$-module let
$l_\pr(N)$ denote its length. Then we put
$$l_\pr(M\otimes_A A_\pr):=\sup\{l_\pr(M_\alpha\otimes_A A_\pr)\}\,,$$
where $M_\alpha$ varies among all finitely generated submodules of $M$. Note that, since $A_\pr$ is flat, $M_\alpha\otimes_A A_\pr$
is still a submodule of $M\otimes_A A_\pr$; furthermore, the length $l_\pr$ is an increasing function on finitely generated torsion
$A_\pr$-modules (partially ordered by inclusion).
\end{itemize}
\end{rems}

\subsubsection{Power series} \label{ssubsecpowser}
In the rest of this section, $A$ will denote a Krull domain and $B:=A[[t]]$ the ring of power series in one variable over $A$.

\begin{prop} \label{p:Rlambda}
Let $A$ be a Krull domain and $\pr\subset A$ a height 1 prime. Then $B$ is also a Krull domain and $\pr B$ is a height 1 prime of $B$.
\end{prop}

\noindent This is well-known (actually, one can prove the analogue even with infinitely many variables: see \cite{gil}). In order to make the
paper as self-contained as possible, and for lack of a suitable reference for the second part of the proposition, we provide a quick proof.

\begin{proof} Let $Q$ be the fraction field of $A$. Since $A$ is Krull, we have $A=\cap A_\gotq$ as $\gotq$ varies among all prime
ideals of height 1. Furthermore, each $A_\gotq$ is a discrete valuation ring: then \cite[VII, \S 3.9, Proposition 8]{Bo} shows that
$A_\gotq[[t]]$ is a unique factorization domain. In particular every $A_\gotq[[t]][t^{-1}]$ is a Krull domain and we get
\begin{equation}  \label{IntersEq}
B=A[[t]]= Q[[t]]\cap\bigcap_\gotq A_\gotq[[t]][t^{-1}]=
Q[[t]]\cap\bigcap_\gotq \bigcap_{\gotP\in S_\gotq} \big(A_\gotq[[t]][t^{-1}]\big)_{\gotP}
\end{equation}
(where $S_\gotq$ denotes the set of height 1 primes in $A_\gotq[[t]][t^{-1}]$). This shows that $B$ is an intersection of
discrete valuation rings. A power series $\lambda=t^h\sum_{i\geq 0}c_it^i\in B$ (with $c_0\neq0$) is a unit in
$A_\gotq[[t]][t^{-1}]$ unless $c_0\in\gotq$ and, in the latter case, $\lambda$ is still a unit in
$(A_\gotq[[t]][t^{-1}])_{\gotP}$ unless it can be divided by the generator of $\gotP$. This proves that $B$ is a Krull domain.

\noindent Since $A_\pr$ is a discrete valuation ring, its maximal ideal $\pr A_\pr$ is principal: let $\pi$ be a uniformizer.
Then $\pi$ is irreducible in $A_\pr[[t]]$, hence it generates a height 1 prime ideal $\gotP:=\pi A_\pr[[t]]=\pr A_\pr[[t]]$.
By the general theory of Krull domains, $\gotP$ corresponds to a discrete valuation $\nu_\gotP$ on the fraction field
$Frac(A_\pr[[t]])$; the restriction of $\nu_\gotP$ to $Q$ is precisely the discrete valuation associated with $\pr$.
Similarly, restricting $\nu_\gotP$ to $Frac(B)$ yields a discrete valuation, with ring of integers $D_\gotP$ and maximal ideal $\gotm_\gotP$.
The ring $D_\gotP$ is the localization of $B$ at $\gotm_\gotP$: hence it is flat over $B$ and, by \cite[VII,
\S 1.10, Proposition 15]{Bo}, the prime ideal
\[ \gotm_\gotP\cap B=\gotP\cap B=\pr B\neq0 \]
has height 1.\end{proof}

\subsection{Pseudo-null $B$-modules} \label{subsecpseudonull} Now assume that $A$ (and hence $B$) is Noetherian.
In this section $P$ will be a pseudo-null $B$-module. We denote by $P_t$ the kernel of multiplication by $t$ and remark that in the exact sequence
\begin{equation}\label{exseqtP} \xymatrix{ P_t \ar@{^(->}[r] & P \ar[r]^t & P \ar@{->>}[r] & P/tP \,,}\end{equation}
$P_t$ and $P/tP$ are finitely generated $B$-modules, because so is $P$. The former ones are also finitely generated as $A$-modules,
because $t$ acts as 0 on them.
Moreover they are torsion $A$-modules (just take two relatively prime elements $f$ and $g$ in $Ann_B(P)$: their projections in
$A$ via $t\mapsto 0$ belong to both $Ann_A(P_t)$ and $Ann_A(P/tP)$ and at least one of them is nonzero since otherwise $t$ would
divide both $f$ and $g$). Therefore the characteristic ideals $Ch_A(P_t)$ and $Ch_A(P/tP)$ are given by Definition \ref{defCharid}
(there is no need for supernatural divisors here) and both of them are nonzero.

\subsubsection{Preliminaries} For $\pr$ a prime of height one in $A$, define
\[ \wh{A_\pr}:=\liminv A_\pr/\pr^n A_\pr\ . \]
By a slight abuse of notation, we shall denote by $\pr$ also the maximal ideals of $A_\pr$ and $\wh{A_\pr}$.
The natural embedding of $A$ into $\wh{A_\pr}$ allows to identify $B$ with a subring of $\wh{A_\pr}[[t]]$.

\begin{lem} \label{l:flat}
The ring $\wh{A_\pr}[[t]]$ is a flat $B$-algebra.
\end{lem}

\begin{proof} Put $S_\pr:=A-\pr$. We claim that $\wh{A_\pr}[[t]]$ is the completion of $S_\pr^{-1}B$ with respect to the ideal generated by
$\pr$ and $t$. This is enough, since formation of fractions and completion of a noetherian ring both generate flat algebras,
and the composition of flat morphisms is still flat. To verify the claim consider the inclusions
\[ A_\pr[t]/(\pr, t)^n \subset S_\pr^{-1}B/(\pr, t)^n \subset \wh{A_\pr}[[t]]/(\pr, t)^n \]
and note that they are preserved by taking the inverse limit with respect to $n$. To conclude observe that
$\liminv A_\pr[t]/(\pr, t)^n=\wh{A_\pr}[[t]]$.
\end{proof}

\noindent The advantage of working over $\wh{A_\pr}[[t]]$ is that one can apply the Weierstrass Preparation Theorem (for a proof see e.g.
\cite[VII, \S3.8, Proposition 6]{Bo}): given $\alpha= \sum a_i t^i \in \wh{A_\pr}[[t]]$ such that not all coefficients are in $\pr$, there
exist $u\in \wh{A_\pr}[[t]]^*$ and a monic polynomial $\beta\in\wh{A_\pr}[t]$ such that $\alpha=u\beta$ (the degree of $\beta$ is equal
to the minimum of the indices $i$ such that $a_i\not\in \pr$).
Actually, as it is going to be clear from the proof of Lemma \ref{fingen} below, we shall need just a weaker form of this statement.

\subsubsection{Characteristic ideals} \label{subssCharId}
Now we deal with the equality between $Ch_A(P_t)$ and $Ch_A(P/tP)$.

\begin{lem}\label{EqLengths}
For any finitely generated torsion $A_\pr$-module $N$ one has the equality of lengths
\[ l_{A_\pr}(N) = l_{\wh{A_\pr}}(N\otimes_{A_\pr}\wh{A_\pr}) \ .\]
\end{lem}

\begin{proof}
Since both $A_\pr$ and $\wh{A_\pr}$ are discrete valuation rings and
\[ A_\pr/\pr^n\simeq\wh{A_\pr}/\pr^n\simeq(A_\pr/\pr^n)\otimes\wh{A_\pr}\ , \]
the statement follows directly from the structure theorem for finitely generated torsion modules over principal ideal domains.
\end{proof}

\begin{lem}\label{fingen} Let $P$ be a pseudo-null $B$-module. Then $P\otimes_B\wh{A_\pr}[[t]]$ is a finitely generated
$\wh{A_\pr}$-module for any height 1 prime ideal $\pr\subset A$.
\end{lem}

\begin{proof}
We consider $P$ as an $A$-module and work separately with primes $\pr$ belonging or not belonging to the support of $P$. If the prime $\pr$ is
not in this support, there is some $r\in Ann_A (P)$ which becomes a unit in $A_\pr\subset \wh{A_\pr}[[t]]$, hence
$P\otimes_B\wh{A_\pr}[[t]]=0$ and the statement is trivial.
Thus, from now on, we assume $\pr \in Supp_A (P)$ (i.e., $Ann_A (P) \subset \pr$). Since $\pr B$ is a height 1 prime ideal
in $B$, the hypothesis on $P$ yields $Ann_B (P)\not\subset \pr B$.
Hence there exists $\alpha \in Ann_B (P) - \pr B$, i.e.,
\[ \alpha = \sum_{i\ge 0} a_i t^i\in Ann_B (P)\qquad ({\rm with\ } a_i\in A\ {\rm for\ any\ }i) \]
with at least one $a_i \not\in \pr$. For such an $\alpha$, let $n$ be the smallest index such that $a_n\notin\pr$.
Then, by \cite[VII, \S 3.8, Proposition 5]{Bo} (which is a key step in the proof of the Weierstrass Preparation Theorem), one has a decomposition
\[ \wh{A_\pr}[[t]]=\alpha \wh{A_\pr}[[t]]\oplus\big( \bigoplus_{i=0}^{n-1} \wh{A_\pr} t^i\big)\ .\]
Now one just uses $P\otimes_B\alpha\wh{A_\pr}[[t]]=\alpha\cdot(P\otimes_B\wh{A_\pr}[[t]])=0$.
\end{proof}

\begin{prop} \label{pseudonull}
Let $P$ be a pseudo-null $B$-module. Then $Ch_A(P_t)=Ch_A(P/tP)$.
\end{prop}

\begin{proof}
By Remark \ref{supernatpseudoisom} and Lemma \ref{EqLengths}, we need to show that
\[ l_{\wh{A_\pr}}(P_t\otimes_A\wh{A_\pr}) = l_{\wh{A_\pr}} ((P/tP)\otimes_A\wh{A_\pr}) \]
for any height 1 prime ideal $\pr$ of $A$. By Lemma \ref{l:flat}, the functor $\otimes_B\wh{A_\pr}[[t]]$ is exact.
Applying it to \eqref{exseqtP}, we get an exact sequence
\begin{equation} \label{exseqtP2}
P_t\otimes_B\wh{A_\pr}[[t]] \iri P\otimes_B\wh{A_\pr}[[t]] {\buildrel t\over\longrightarrow}
P\otimes_B\wh{A_\pr}[[t]] \sri (P/tP)\otimes_B\wh{A_\pr}[[t]]\ . \end{equation}
Lemma \ref{fingen} shows that all terms of \eqref{exseqtP2} are finitely generated $\wh{A_\pr}$-modules.
Hence, the first and last term of the sequence have the same length. Finally, just observe that if $N$ is a $B$-module annihilated by $t$ then
\[ N\otimes_B\wh{A_\pr}[[t]] = N\otimes_A\wh{A_\pr}\ .\]
\end{proof}

\begin{exem}
If $P$ happens to be finitely generated over $A$ then the statement of the proposition is obvious.
We give a few examples of pseudo-null $B:=\Z_p[[s,t]]$-modules which are not finitely generated as $A:=\Z_p[[s]]$-modules,
providing non-trivial examples in which the above theorem applies.
However we remark that the main consequence of Lemma \ref{fingen} is exactly the fact that we can ignore the issue of checking
whether a pseudo-null $B$-module is finitely generated over $A$ or not.
\begin{itemize}
\item[{\bf 1.}] $P=B/(p,s)\,$. Then $P\simeq \F_p[[t]]$ is not finitely generated over $A$. In this case $P_{t}=0$ and
$P/tP\simeq \F_p$ (both $A$-pseudo-null), so that
\[ Ch_A(P_{t}) = Ch_A(P/tP) = A\ .\]
\item[{\bf 2.}] $P=B/(s,pt)\,$. Then $P\simeq \Z_p[[t]]/(pt)$ is not finitely generated over $A$ and elements in $P$ can be written as
\[ m = \sum_{i\ge 0} a_i t^i \quad a_0\in \Z_p\ {\rm and}\ a_i\in \{0,...,p-1\}\ \forall\,i\ge 1\ .\]
Moreover
\[ P_{t} = p \Z_p[[t]]/(pt) \simeq p\Z_p \simeq \Z_p \simeq A/(s) \]
and
\[ P/tP = \Z_p[[s,t]]/(t,s,pt) \simeq \Z_p \simeq A/(s) \ ,\]
so both have characteristic ideal $(s)$ (as $A$-modules).
\item[{\bf 3.}] With $P=B/(p,st)$, a similar reasoning shows that $Ch_A(P/tP)=Ch_A(P_t)=(p)$.
\end{itemize}
\end{exem}

\begin{rem}\label{failsintorsion} The hypothesis that $P$ is pseudo-null is necessary: if $M$ is a torsion $B$-module then
it is not true, in general, that $Ch_A(M_t)=Ch_A(M/tM)$. We give an easy example: let again $B=\Z_p[[s,t]]$ with $A=\Z_p[[s]]$, and
consider $M=\Z_p[[s,t]]/(p^2+s+t)$, which is a torsion $B$-module. Observe that $M_t$ is trivial (so $Ch_A(M_t)=A$) and
\[ M/tM=\Z_p[[s,t]]/(t,p^2+s)\simeq A/(p^2+s) \]
has characteristic ideal over $A$ equal to $(p^2+s)$.
Moreover $Ch_A(M/tM)=(p^2+s)$ is the image of $Ch_B(M)=(p^2+s+t)$ under the projection $\pi\colon B\rightarrow A$, $t\mapsto 0$.
Hence, in this case,
\[ Ch_A(M_t)\, \pi \left(Ch_B(M)\right) = Ch_A(M/tM) \]
which anticipates the general formula of Proposition \ref{CharId1}.
\end{rem}

As mentioned in the Introduction, the following proposition will be crucial in the study of characteristic ideals for
Iwasawa modules under descent.

\begin{prop}\label{CharId1}
Let $\pi\colon B\rightarrow A$ be the projection given by $t\mapsto 0$ and let $M$ be a finitely generated torsion $B$-module. Then
\begin{equation} \label{descentEq} Ch_A(M_t)\, \pi\left(Ch_B(M)\right) = Ch_A(M/tM) \ .\end{equation}
Moreover,
\[ Ch_A(M_t)=0\Longleftrightarrow \pi\left(Ch_B(M)\right) =0\Longleftrightarrow Ch_A(M/tM)= 0 \]
and in this case $M_t$ and $M/tM$ are $A$-modules of the same rank.
\end{prop}

\begin{proof} Recall that the structure theorem \eqref{StrThEq} provides a pseudo-iso\-mor\-phism between $M$ and its associated
elementary module $E(M)$. As noted above, being pseudo-isomorphic is an equivalence relation for torsion modules: therefore
one has a (non-canonical) sequence
\[ \xymatrix{ E(M)\, \ar@{^(->}[r]  & M \ar@{->>}[r]  & P  } \]
where $P$ is pseudo-null over $B$ and the injectivity on the left comes from the fact that elementary modules have no nontrivial
pseudo-null submodules (just use the valuation on $B_\pr$ to check that the annihilator of any $x\in B/\pr^e-\{0\}$ must be
contained in $\pr$). The snake lemma sequence coming from the diagram
\[ \xymatrix{ E(M) \ar@{^(->}[r] \ar[d]^{t} & M \ar@{->>}[r] \ar[d]^{t} & P \ar[d]^{t} \\
E(M) \ar@{^(->}[r] & M \ar@{->>}[r]  & P } \]
reads as
\begin{equation}\label{SnLemEq}
E(M)_t \iri M_t \longrightarrow P_t \longrightarrow E(M)/tE(M) \longrightarrow M/tM \sri P/tP \ . \end{equation}
As we remarked at the beginning of Section \ref{subsecpseudonull}, both $P_t$ and $P/tP$ are finitely generated torsion $A$-modules.
It is also easy to see that all modules in the sequence \eqref{SnLemEq} are finitely generated over $A$.
Now observe that $(B /\pr^e)_t$ is zero if $\pr\neq(t)$ and isomorphic to $A$ if $\pr=(t)$;
similarly, $(B/\pr^e)/t(B/\pr^e)$ is either pseudo-null or isomorphic to $A$. Thus, putting $E(M)= \oplus B /\pr_i^{e_i}$,
we find $E(M)_t\simeq A^r$ and
\[ E(M)/tE(M)= \oplus B /(\pr_i^{e_i},t) \simeq \oplus A/(\pi(\pr_i)^{e_i}) \simeq A^r\oplus \bullet \ , \]
where $r:=\#\{i\mid\pr_i=t B\}$ and $\bullet$ is a pseudo-null $B$-module.
Moreover \eqref{SnLemEq} shows that $E(M)/tE(M)$ is $A$-torsion if and only if $M/tM$ is $A$-torsion and $E(M)_t$ is $A$-torsion
if and only if $M_t$ is $A$-torsion.
Therefore we have two cases:\begin{itemize}
\item[{\bf 1.}] if $r>0$, then $(t)$ divides $Ch_B(M)$, so $\pi(Ch_B(M))=0$ and, since $M_t$ and $M/tM$ are not
$A$-torsion, $Ch_A(M_t)=Ch_A(M/tM)=0$ as well (the statement on $A$-ranks is immediate from \eqref{SnLemEq}: e.g., apply
the exact functor $\otimes_A Frac(A)$);
\item[{\bf 2.}] if $r=0$, then, because of the equivalent conditions above, all the characteristic ideals involved in
\eqref{descentEq} are nonzero; moreover we have
\[ Ch_A(E(M)/tE(M)) = \pi(Ch_B(E(M))) = \pi(Ch_B(M)) \]
and \eqref{descentEq} follows from the sequence \eqref{SnLemEq}, Proposition \ref{pseudonull} and the multiplicativity
of characteristic ideals.
\end{itemize}
\end{proof}

\begin{cor}\label{EqVerCharId}
In the above setting assume that $M/tM$ is a finitely generated torsion $A$-module. Then
$M$ is a pseudo-null $B$-module if and only if $Ch_A(M_t)= Ch_A(M/tM)\,$. Moreover if $M/tM \sim_A 0$, then $M \sim_B 0$.
\end{cor}

\begin{proof} The ``only if'' part is provided by Proposition \ref{pseudonull}. For the ``if'' part we assume the equality of
characteristic ideals (which are nonzero by hypothesis). By \eqref{descentEq} we have $\pi(Ch_B(M))=A$, hence there is
some $f\in Ch_B(M)$ such that $\pi(f)=1$. But then $f=\sum_{i\ge 0} c_i t^i$
with $c_0=1$, which is an obvious unit in $B=A[[t]]$. Therefore $Ch_B(M)=B\,$, i.e., $M$ is pseudo-null over $B$.
For the last statement just note that $Ch_A(M/tM)=A$ yields $Ch_A(M_t)\, \pi(Ch_B(M))=A$, so $Ch_A(M_t)=A$ as well.
\end{proof}

\begin{rems}\label{RemPR}
\begin{itemize}  \item[{}]
\item[{\bf 1.}] When $R\simeq \Z_p[[t_1,\ldots,t_d]]$ (i.e., the Iwasawa algebra for a $\Z_p^d$-ex\-ten\-sion of global fields),
the statement of the previous corollary appears in \cite[Lemme 4]{PR}. Note
anyway that the proof given there relies on the choice of a $\Z_p^{d-1}$-subextension (i.e., on the strategy
mentioned in Remark \ref{IntroRem}).
\item[{\bf 2.}] The possibility of lifting pseudo-nullity from $M/tM$ to $M$ has been used to prove some instances
of Greenberg's Generalized Conjecture
(for statement and examples see, e.g., \cite{Ba}, \cite{Ba2} and \cite{Oz}).
\end{itemize}
\end{rems}

\subsection{Pro-characteristic ideals}\label{SecProCharId}
We can now define an analogue of characteristic ideals for finitely generated modules over certain non-noetherian Krull
domains $\L$. We need $\L$ to be the inverse limit of noetherian Krull domains and we limit ourselves to finitely generated
modules because characteristic ideals are usually defined only for them.

\noindent Let $\{\L_d\}_{d\geq 0}$ be an inverse system of noetherian Krull domains such that
\[ \L_d \simeq \L_{d+1}/\pr_{d+1}\ \  {\rm and}\ \  \L_{d+1} \simeq \il{n} \L_{d+1}/\pr_{d+1}^n\ \
{\rm for\ any\ }d\geq 0 \]
($\pr_{d+1}$ a principal prime ideal of $\L_{d+1}$ of height 1).
Let $\L:=\displaystyle{\il{d} \L_d}$ and note that, by hypothesis, $\L_{d+1}\simeq \L_d[[t_{d+1}]]$, where the
variable $t_{d+1}$ corresponds to a generator of the ideal $\pr_{d+1}\,$.
Take a finitely generated $\L$-module $M$ which can be written as the inverse limit of $\L_d$-modules
$M = \displaystyle{\il{d} M_d}$ (all the relevant arithmetic applications to Iwasawa modules satisfy
this requirement).

\begin{thm}\label{MainThm}
With the above notations if, for any $d \gg 1$, \begin{itemize}
\item[{\bf 1.}] $(M_d)_{t_d}$ (the $\pr_d$-torsion submodule of $M_d\,$) is a pseudo-null $\L_{d-1}$-module;
\item[{\bf 2.}] $Ch_{\L_{d-1}}(M_d/t_dM_d)\subseteq Ch_{\L_{d-1}}(M_{d-1})$,
\end{itemize}
then the {\em pro-characteristic ideal} of $M$ over $\L$ is well defined as
\[ \wt{Ch}_\L(M):= \il{d}\, (\pi^\L_{\L_d})^{-1}(Ch_{\L_d}(M_d)) \]
(where $\pi^\L_{\L_d}\,: \L\rightarrow \L_d$ is the natural map provided by the inverse limit defining $\L$).
\end{thm}

\begin{proof} We can assume that the $M_d$ are torsion $\L_d$-modules (at least for $d\gg 0$), otherwise
the $Ch_{\L_d}(M_d)$ are zero and there is nothing to prove.
By Proposition \ref{CharId1}, applied to $A=\L_{d-1}\,$, $B=\L_d$ and $M=M_d\,$, we get
\[ Ch_{\L_{d-1}}((M_d)_{t_d})\pi^{\L_d}_{\L_{d-1}}(Ch_{\L_d}(M_d))=Ch_{\L_{d-1}}(M_d/t_d M_d) \ .\]
For $d\gg 1$ the hypotheses yield
\[ \pi^{\L_d}_{\L_{d-1}}(Ch_{\L_d}(M_d))\subseteq Ch_{\L_{d-1}}(M_{d-1}) \ ,\]
which shows that the generators of the ideals $Ch_{\L_d}(M_d)$ form a coherent sequence with respect to the maps defining $\L$.
Hence this sequence defines an element in $\L$ which can be considered as a generator for the pro-characteristic ideal
\[ \wt{Ch}_\L(M):= \il{d}\, (\pi^\L_{\L_d})^{-1}(Ch_{\L_d}(M_d)) \ .\]
\end{proof}

Our pro-characteristic ideal maintains two classical properties of characteristic ideals.

\begin{cor}\label{CorMainThm}
Let $M$, $M'$ and $M''$ be finitely generated $\L$-modules which verify the hypotheses of Theorem \ref{MainThm}.\begin{itemize}
\item[{\bf 1.}] The pro-characteristic ideals are multiplicative, i.e., if there is an exact sequence
\begin{equation}\label{MTEq1}
\xymatrix{ M' \ar@{^(->}[r] & M \ar@{->>}[r] & M'' } \ ,
\end{equation}
then
\[ \wt{Ch}_\L(M)=\wt{Ch}_\L(M')\wt{Ch}_\L(M'')\ .\]
\item[{\bf 2.}] $\wt{Ch}_\L(M)\neq 0$ if and only if $M_d$ is a finitely generated torsion $\L_d$-module
for $d\gg 0$.
\end{itemize}
\end{cor}

\begin{proof}
{\bf 1.} For any $d\geq 0$ we have exact sequences (arising from \eqref{MTEq1}\,)
\[ \xymatrix{ M'_d \ar@{^(->}[r] & M_d \ar@{->>}[r] & M''_d } \ , \]
for which the equality $Ch_{\L_d}(M_d)=Ch_{\L_d}(M'_d)Ch_{\L_d}(M''_d)$ holds. The previous theorem allows to take limits
on both sides maintaining the equality.

\noindent {\bf 2.} Obvious.
\end{proof}

\begin{rem}\label{RemMainThm}
In the previous corollary it is enough to assume that $M'$ and $M''$ verify the hypotheses of Theorem \ref{MainThm}. Indeed,
using the snake lemma exact sequence
\[ (M'_d)_{t_d}\iri (M_d)_{t_d} \rightarrow (M''_d)_{t_d} \rightarrow M'_d/t_dM'_d \rightarrow M_d/t_dM_d
\sri M''_d/t_dM''_d \ , \]
one immediately has that
\[ (M'_d)_{t_d}\ {\rm and}\ (M''_d)_{t_d} \sim_{\L_{d-1}}0 \Longrightarrow (M_d)_{t_d} \sim_{\L_{d-1}}0 \]
and
\[ \begin{array}{lcl} Ch_{\L_{d-1}}(M_d/t_dM_d) & = & Ch_{\L_{d-1}}(M'_d/t_dM'_d)Ch_{\L_{d-1}}(M''_d/t_dM''_d) \\
\ & \subseteq & Ch_{\L_{d-1}}(M'_{d-1})Ch_{\L_{d-1}}(M''_{d-1}) = Ch_{\L_{d-1}}(M_{d-1})\ . \end{array}\]
\end{rem}

\section{Class groups in global fields}\label{GlobalClGrSec}
For the rest of the paper we adjust our notations a bit to be more consistent with the usual ones in Iwasawa theory.
We fix a prime number $p$ and a global field $F$ (note that for now we are not making any assumption on the
characteristic of $F$). For any finite extension $E/F$ let $\calm(E)$ be the $p$-adic completion of the group of divisor classes of $E$, i.e.,
\[ \calm(E):=(E^*\backslash\idl_E/\Pi_v \ol_{E_v}^*) \otimes \Z_p \]
where $\idl_E$ is the group of finite ideles of $E$, $v$ varies over all non-archimedean places of $E$ and
$\ol_{E_v}$ is the ring of integers of the completion of $E$ at $v$. When $\call/F$ is an infinite extension,
we put $\calm(\call):=\liminv\calm(E)$ as $E$ runs among finite subextensions of $\call/F$ (the limit being taken with respect
to norm maps). Class field theory yields a canonical isomorphism
\begin{equation} \label{ArtinMapEq} \calm(E)\stackrel{\sim}{\longrightarrow} X(E):=\Gal(L(E)/E) \ ,\end{equation}
where $L(E)$ is the maximal unramified abelian pro-$p$-extension of $E$. Passing to the limit shows that \eqref{ArtinMapEq} is
still true for infinite extensions.

\noindent Finally, for any infinite Galois extension $\call/F$, let $\L(\call):=\Z_p[[\Gal(\call/F)]]$ be the associated Iwasawa algebra.
We shall be interested in the situation where $\Gal(\call/F)$ is an abelian $p$-adic Lie group: in this case, both $\calm(\call)$
and $X(\call)$ are $\L(\call)$-modules (the action of $\Gal(\call/F)$ on $X(\call)$ is the natural one via inner automorphisms of
$\Gal(L(\call)/F)\,$) and these structures are compatible with the isomorphism \eqref{ArtinMapEq}.
Furthermore, if $\Gal(\call/F)\simeq\Z_p^d$ then $\L(\call)\simeq\Z_p[[t_1,..,t_d]]$ is a Krull domain.

\begin{lem} \label{RamLemma}
Let $\calf/F$ be a $\Z_p^d$-extension, ramified only at finitely many places. If $d>2$, one can always find a $\Z_p$-subextension
$\calf_1/F$ such that none of the ramified places splits completely in $\calf_1$.
\end{lem}

\begin{proof} Let $S$ denote the set of primes of $F$ which ramify in $\calf$ and, for any place $v$ in $S$ let $D_v\subset\Gal(\calf/F)=:\Gamma$
be the corresponding decomposition group. Getting $\calf_1$ amounts to finding $\alpha\in\Hom(\Gamma,\Z_p)$
such that $\alpha(D_v)\neq0$ for all $v\in S$. By hypothesis, for such $v$'s the vector spaces $D_v\otimes\Q_p$
are non-zero, hence their annihilators are proper subspaces of $\Hom(\Gamma\otimes\Q_p,\Q_p)$ and since a
$\Q_p$-vector space cannot be union of a finite number of proper subspaces, we deduce that the required $\alpha$ exists.
\end{proof}

\noindent The following lemma is mostly a restatement of  \cite[Theorem 1]{Gr1}.

\begin{lem}\label{GGCLemma}
Let $\calf/F$ be a $\Z_p^d$-extension, ramified only at finitely many places, and $\calf'\subset\calf$ a $\Z_p^{d-1}$-subextension, with $d>2$.
Let $I$ be the kernel of the natural projection $\L(\calf)\rightarrow\L(\calf')$. Then $X(\calf)/IX(\calf)$ is a finitely generated
torsion $\L(\calf')$-module and $X(\calf)$ is a finitely generated torsion $\L(\calf)$-module. This holds also for $d=2$, provided
that no ramified place in $\calf/F$ is totally split in $\calf'$.
\end{lem}

\begin{proof} The idea is to proceed by induction on $d$. Choose a filtration
$$F=:\calf_0\subset\calf_1\subset\dots\subset\calf_{d-1}:=\calf'\subset\calf_d:=\calf$$
so that $\Gal(\calf_i/\calf_{i-1})\simeq\Z_p$ for all $i$ and no ramified place in $\calf/F$ is totally split in $\calf_1$
(by Lemma \ref{RamLemma}, this can always be achieved when $d>2$).

\noindent Now one proceeds as in \cite[Theorem 1]{Gr1}. Namely, a standard argument yields that a $\L(\calf_i)$-module $M$ is in
${\bf Fgt}_{\L(\calf_i)}$ if $M/I^i_{i-1}M$ is in ${\bf Fgt}_{\L(\calf_{i-1})}$ (where $I^i_{i-1}$ is the kernel of the
projection $\L(\calf_i)\rightarrow\L(\calf_{i-1})$\,) and Greenberg's proof shows that $X(\calf_{i-1})\in {\bf Fgt}_{\L(\calf_{i-1})}$
implies $X(\calf_i)/I^i_{i-1}X(\calf_i)\in {\bf Fgt}_{\L(\calf_{i-1})}$. So it is enough to prove that $X(\calf_1)$ is a finitely
generated torsion $\L(\calf_1)$-module. This follows from Iwasawa's classical proof (\cite{Iw}, exposed e.g. in \cite{Se};
the function field version can be found in \cite{lz}).
\end{proof}

\begin{rems}
\begin{itemize}  \item[{}]
\item[{\bf 1.}] In a $\Z_p$-extension of a global field, only places with residual characteristic $p$ can ramify: thus the finiteness
hypothesis on the ramification locus is automatically satisfied unless $char(F)=p$. Note, however, that in the latter case this
hypothesis is needed (see, e.g. \cite[Remark 4]{GK}).
\item[{\bf 2.}] Among all $\Z_p$-extensions of $F$ there is a distinguished one, namely, the cyclotomic $\Z_p$-extension $F_{cyc}$ if
$F$ is a number field and the arithmetic $\Z_p$-extension $F_{arit}$ (arising from the unique $\Z_p$-extension of the constant field)
if $F$ is a function field. The condition on $\calf'$ (when $d=2$) is satisfied if it contains either $F_{cyc}$ or $F_{arit}$.
\item[{\bf 3.}] For $d=1$, we have $\calf'=F$ and $\L(\calf')=\Z_p$. Thus the analogue of Lemma \ref{GGCLemma} would state that
$X(\calf)/IX(\calf)$ is finite. This holds quite trivially if $F$ is a global function field and $\calf=F_{arit}$ (note also that
if $char(F)=\ell\neq p$ then $F_{arit}$ is the only $\Z_p$-extension of $F$, see e.g.~\cite[Proposition 4.3]{BL2}). In this case
the maximal abelian extension of $F$ contained in $L(\calf)$ is exactly $L(F)$, hence $X(\calf)/IX(\calf)\simeq \Gal(L(F)/F_{arit})$
which is known (e.g. by class field theory) to be finite.
\end{itemize}
\end{rems}

\subsection{Iwasawa theory for class groups in function fields}\label{ClGrSec}
In this section $F$ will be a global function field of characteristic $p$ and $F_{arit}$ its arithmetic $\Z_p$-extension as defined above.
Let $\calf/F$ be a $\Z_p^\infty$-extension unramified outside a finite set of places $S$, with $\G:=\Gal(\calf/F)$ and associated Iwasawa
algebra $\L:=\L(\calf)$. We fix a $\Z_p$-basis $\{\gamma_i\}_{i\in\N}$ for $\G$ and for any $d\geq 0$ we let $\calf_d\subset\calf$ be the
fixed field of $\{\gamma_i\}_{i>d}$. Also, we assume that our basis is such that no place in $S$ splits completely in $\calf_1$
(Lemma \ref{RamLemma} shows that there is no loss of generality in this assumption).

\begin{rem}
If $\calf$ contains $F_{arit}$ we can take the latter as $\calf_1$. The additional hypothesis on $\calf_1$ appears also in
\cite[Theorem 1.1]{KLT1}: the authors enlarge the set $S$ and the extension $\calf_d$ in order to get a
$\Z_p$-extension verifying that hypothesis and use this to get a monomial Stickelberger element. This is a crucial step in the
proof of the Main Conjecture provided in \cite{blt09A}.
\end{rem}

\noindent For notational convenience, let $t_i:=\gamma_i-1$. Then the subring  $\Z_p[[t_1,\dots, t_d]]$ of $\L$ is canonically isomorphic
to $\L(\calf_d)$ and, by a slight abuse of notation, the two shall be identified in the following. In particular, for any
$d\geq1$ we have $\L(\calf_d)=\L(\calf_{d-1})[[t_d]]$ and we can apply the results of Section \ref{CharIdSec}.
We shall denote by $\pi^d_{d-1}$ the canonical projection $\L(\calf_d)\rightarrow\L(\calf_{d-1})$ with kernel $I^d_{d-1}=(t_d)$
(the augmentation ideal of $\calf_d/\calf_{d-1}\,$) and by $\G^d_{d-1}$ the group $\Gal(\calf_d/\calf_{d-1})$.

For two finite extensions $L\supset L'\supset F$, the degree maps $\deg_L$ and $\deg_{L'}$ fit into the commutative diagram
(with exact rows)
\begin{equation}\label{CommDiagDeg}
\xymatrix{\cala(L)\ar@{^(->}[r] \ar[d]^{N^L_{L'}}& \calm(L) \ar@{->>}[rr]^{\deg_L}\ar[d]^{N^L_{L'}} & & \Z_p \ar[d] \\
\cala(L')\ar@{^(->}[r] & \calm(L') \ar@{->>}[rr]^{\deg_{L'}} & & \Z_p\,, } \end{equation}
where $N^L_{L'}$ denotes the norm and the vertical map on the right is multiplication by $[\mathbb{F}_L:\mathbb{F}_{L'}]$
(the degree of the extension between the fields of constants). For an infinite extension $\call/F$ contained in $\calf$, taking
projective limits one gets an exact sequence
\begin{equation}\label{degseq}
\xymatrix{ \cala(\call)\ar@{^(->}[r] & \calm(\call) \ar@{->}[rr]^{\deg_{\call}} & & \Z_p } \ .
\end{equation}

\begin{rem} The map $\deg_{\call}$ above becomes zero exactly when $\call$ contains the unramified $\Z_p$-subextension $F_{arit}$. \end{rem}

\noindent By \eqref{ArtinMapEq}, Lemma \ref{GGCLemma} shows that $\calm({\calf_d})$ is a finitely generated torsion
$\L(\calf_d)$-module, so the same holds for $\cala(\calf_d)$. Hence, by Proposition \ref{CharId1}, one has, for all $d\geq1$,
\begin{equation}\label{CharIdA0} Ch_{\L(\calf_{d-1})}(\cala(\calf_d)_{t_d})\,\pi^d_{d-1}(Ch_{\L(\calf_d)}(\cala(\calf_d))) =
Ch_{\L(\calf_{d-1})}(\cala(\calf_d)/t_d\cala(\calf_d))
\end{equation}
and note that
\[ \cala(\calf_d)_{t_d}=\cala(\calf_d)^{\G^d_{d-1}}\ \ ,\ \  \cala(\calf_d)/t_d\cala(\calf_d)=
\cala(\calf_d)/I^d_{d-1}\cala(\calf_d)\ .\]

\noindent Consider the following diagram
\begin{equation}
\xymatrix{ \cala(\calf_d) \ar@{^(->}[r]\ar[d]^{t_d} & \calm(\calf_d) \ar@{->>}[r]^{\deg}\ar[d]^{t_d} &
\Z_p \ar[d]^{t_d} \\
\cala(\calf_d) \ar@{^(->}[r] & \calm(\calf_d) \ar@{->>}[r]^{\deg} & \Z_p }\end{equation}
(note that the vertical map on the right is 0) and its snake lemma sequence
\begin{equation} \label{snake0}
\xymatrix{ \cala(\calf_d)^{\G^d_{d-1}} \ar@{^(->}[r] & \calm(\calf_d)^{\G^d_{d-1}} \ar[r]^{\deg} &
\Z_p \ar[d] \\
\Z_p & \calm(\calf_d)/I^d_{d-1}\calm(\calf_d) \ar@{->>}[l]_{\deg\ \quad\ } &  \cala(\calf_d)/I^d_{d-1}\cala(\calf_d) \ar[l]\ .}
\end{equation}

\noindent For $d\ge 2$ (which implies that $\Z_p$ is a torsion $\L(\calf_{d-1})$-module), \eqref{snake0} and Lemma \ref{GGCLemma} show
that $\cala(\calf_d)/I^d_{d-1}\cala(\calf_d)$ is in ${\bf Fgt}_{\L(\calf_{d-1})}$ as well. By Proposition \ref{CharId1} it
follows that no term in \eqref{CharIdA0} is trivial.

\subsubsection{Totally ramified extensions and the Main Conjecture} \label{PrRamCGSubs}
The main examples we have in mind are extensions satisfying the following

\begin{ass}\label{AssRam} The (finitely many) ramified places of $\calf/F$ are totally ramified. \end{ass}

\noindent In what follows an extension satisfying this assumption will be called a {\it totally ramified extension}.
A prototypical example is the $\mathfrak{a}$-cyclotomic extension of $\F_q(T)$ generated by the $\mathfrak{a}$-torsion of the Carlitz
module ($\mathfrak{a}$ an ideal of $\F_q[T]$, see e.g. \cite[Chapter 12]{Ro}).
As usual in Iwasawa theory over number fields, most of the proofs will work (or can be adapted) simply assuming that ramified primes
are totally ramified in $\calf/\calf_e$ for some $e\geq 0$, but, in the function field setting, one would need
some extra hypothesis on the behaviour of these places in $\calf_e/F$.

Under this assumption any $\Z_p$-subextension can play the role of $\calf_1\,$. Moreover $\calm(\calf)$ is defined using norm maps and
norms are surjective on class groups in totally ramified extensions, so
\[ \calm(\calf_d)=N^\calf_{\calf_d}(\calm(\calf)):=\calm(\calf)_d \ \ {\rm and}\ \ \calm(\calf)=\il{d} \calm(\calf)_d=
\il{d} \calm(\calf_d) \]
(in the notations of Theorem \ref{MainThm}). The same holds for the modules $\cala(\calf)$ and $\cala(\calf_d)$.

Let $L^0(\calf_{d-1})$ be the maximal abelian extension of $\calf_{d-1}$ contained in $L(\calf_d)$, we
have
\[ \calf_d L(\calf_{d-1})\subseteq L^0(\calf_{d-1})\ \ {\rm and}\ \
\Gal(L(\calf_d)/L^0(\calf_{d-1}))=I^d_{d-1} \calm(\calf_d) \]
(see \cite[Lemma 13.14]{wash}). Galois theory provides a surjection
\[ \Gal(L^0(\calf_{d-1})/\calf_d) \sri \Gal(\calf_dL(\calf_{d-1})/\calf_d)\ ,\]
i.e.,
\[ \calm(\calf_d)/I^d_{d-1} \calm(\calf_d) \sri \calm(\calf_{d-1}) \ ,\]
which yields
\begin{equation}\label{Hyp2Eq}
Ch_{\L_{d-1}}(\calm(\calf_d)/I^d_{d-1}\calm(\calf_d))\subseteq Ch_{\L_{d-1}}(\calm(\calf_{d-1}))\ .
\end{equation}
The same relation holds for the characteristic ideals of the $\cala(\calf_d)$ for $d\geq 3$, because of \eqref{snake0}.
In particular if we have only one ramified prime, the surjection above is an isomorphism (just adapt the proof of \cite[Lemma 13.15]{wash})
and \eqref{Hyp2Eq} is an equality. This takes care of hypothesis {\bf 2} in Theorem \ref{MainThm}.

A little modification of the proof of \cite[Lemma 5.7]{BBL} (note that \cite[Lemmas 5.4 and 5.6]{BBL} still hold
in the present setting), shows that elements of $\calm(\calf_d)^{\G^d_{d-1}}$ are represented by divisors supported on ramified primes.
Hence $\calm(\calf_d)^{\G^d_{d-1}}$ (and $\cala(\calf_d)^{\G^d_{d-1}}\,$) are finitely generated $\Z_p$-modules, i.e.,
pseudo-null $\L(\calf_{d-1})$-modules for $d\geq 3$. From \eqref{CharIdA0} we obtain

\begin{cor}\label{ClGrFinCor}
Let $\calf_d$ be a $\Z_p^d$-extension of $F$ contained in a totally ramified extension. Then, for any
$\Z_p^{d-1}$-extension $\calf_{d-1}$ contained in $\calf_d\,$, one has
\begin{equation}\label{CharIdA1}
\begin{array}{lcl} \pi^d_{d-1}(Ch_{\L(\calf_d)}(\cala(\calf_d))) & = &
Ch_{\L(\calf_{d-1})}(\cala(\calf_d)/I^d_{d-1}\cala(\calf_d)) \\
\ & \subseteq & Ch_{\L_{d-1}}(\cala(\calf_{d-1})) \ .\end{array} \end{equation}
\end{cor}

Hence the modules $\cala(\calf_d)$ verify the hypotheses of Theorem \ref{MainThm} and we can define

\begin{df}\label{CharIdClGr} Let $\mathcal{F}/F$ be a totally ramified $\Z_p^\infty$-extension.
The \emph{pro-characteristic ideal} of $\cala(\calf)$ is
\[ \widetilde{Ch}_{\Lambda}(\cala(\calf)):=\il{\calf_d}\, (\pi_{\calf_d})^{-1}(Ch_{\Lambda(\calf_d)}(\cala(\calf_d)))\ .\]
\end{df}

\begin{rem}
Definition \ref{CharIdClGr} only depends on the extension $\calf/F$ and not on the filtration of $\Z_p^d$-extension we choose inside it.
Indeed take two different filtrations $\{\calf_d\,\}$ and $\{\calf'_d\,\}$ and define a new filtration containing both by putting
\[ \calf''_0:=F\quad{\rm and}\quad \calf''_n=\calf_n\calf'_n \quad \forall\,n\ge 1 \]
(note that $\calf''_n$ is not, in general, a $\Z_p^n$-extension and $\calf''_n/\calf''_{n-1}$ is a $\Z_p^i$-extension with $0\le i \le 2$,
but these details are irrelevant for the limit process we need here). By Corollary \ref{ClGrFinCor}, the limits of the
characteristic ideals of the filtrations we started with coincide with the limit on the filtration $\{\calf''_n\,\}$.
This answers questions $a$ and $b$ of \cite[Remark 5.11]{BBL}: we had a similar definition there but it was based on
the particular choice of the filtration.
\end{rem}

We recall that, in \cite[Theorem 3.1]{blt09} (and \cite{blt09A}), the authors prove an Iwasawa Main Conjecture (IMC)
at ``finite level'', which (in our simplified setting and notations) reads as

\begin{equation}\label{MCFinLev}
Ch_{\Lambda(\calf_d)}(\cala(\calf_d)) = (\theta_{\calf_d/F,S}) \ ,\end{equation}
where $\theta_{\calf_d/F,S}$ is the classical Stickelberger element (defined e.g. in \cite[Section 5.3]{BBL}).
By \cite[Proposition IV.1.8]{tate}, the elements $\theta_{\calf_d/F,S}$ form a coherent sequence with respect to the maps
$\pi^d_e\,$, so, taking inverse limits in \eqref{MCFinLev}, one obtains

\begin{cor}[IMC in non-noetherian algebras]\label{IMCClassGr}
In the previous setting we have
\[ \widetilde{Ch}_{\L}(\cala(\calf)) = \il{\calf_d}\, (\theta_{\calf_d/F,S}) :=(\theta_{\calf/F,S})  \ ,\]
as ideals in $\Lambda$.
\end{cor}

\noindent More details on the statement and its proof (now independent from the filtration $\{\calf_d\}_{d\geq 0}\,$)
can be found in \cite[Section 5]{BBL}.

\begin{rem}
A different approach, using a more natural filtration of global function fields for the Carlitz $\pr$-cyclotomic extension
of $\F_q(T)$ and Fitting ideals of class groups, will be carried out in \cite{ABBL}. It leads to a similar version
of the Iwasawa Main Conjecture in the algebra $\L$, but it has the advantage of having more direct and relevant
arithmetic applications (see \cite[Section 6]{ABBL}).
\end{rem}

\end{document}